\documentclass[10pt]{amsart}
\usepackage{hyperref}
\usepackage[bottom]{footmisc}
\usepackage{graphicx}
\usepackage{enumitem}                               
\usepackage{amssymb}
\usepackage{amsthm}

\newtheorem{theorem}{Theorem}

\newtheorem{corollary}[theorem]{Corollary}

\newtheorem{theoremn}{Theorem}[section] 
\newtheorem{lemman}[theoremn]{Lemma}
\newtheorem{propositionn}[theoremn]{Proposition}
\newtheorem{conjecturen}[theoremn]{Conjecture}

\newtheorem{questionn}[theoremn]{Question}

\newtheorem*{theorem*}{Theorem}
\newtheorem*{lemma*}{Lemma}

\theoremstyle{definition}

\newtheorem{remark}[theoremn]{Remark}

\newtheoremstyle{definition*}% name of the style to be used
{\topsep}% measure of space to leave above the theorem. E.g.: 3pt
{\topsep}% measure of space to leave below the theorem. E.g.: 3pt
{}% name of font to use in the body of the theorem
{0pt}% measure of space to indent
{\bfseries}% name of head font
{.}% punctuation between head and body
{ }% space after theorem head; " " = normal interword space
{\thmname{#1}\thmnumber{ #2}\thmnote{ (#3)}}

\theoremstyle{definition*}
\newtheorem*{definition*}{Definition}
\newtheorem*{remark*}{Remark}
\newtheorem*{claim*}{Claim}

\begin{document}
	
	\baselineskip 13.2pt
	
	\title[ ]{on the influence of the fixed points of an automorphism to the structure of a group }

	\author{{M.Yas\.{I}r} K{\i}zmaz }
	\address{Department of Mathematics, Bilkent University, 06800 
		Bilkent, Ankara, Turkey}
	
	\email{yasirkizmaz@bilkent.edu.tr}
	\maketitle
	
	\begin{abstract}
			Let $\alpha$ be a	coprime automorphism  of  a group $G$ of prime order and let $P$ be an $\alpha$-invariant Sylow $p$-subgroup of $G$. Assume that $p\notin \pi(C_G(\alpha))$. Firstly, we prove that  $G$ is $p$-nilpotent if and only if  $C_{N_G(P)}(\alpha)$ centralizes $P$. In the case that $G$ is  $Sz(2^r)$ and $PSL(2,2^r)$-free where $r=|\alpha|$, we show that $G$ is $p$-closed if and only if $C_G(\alpha)$ normalizes $P$. As a consequences of these two results, we obtain that $G\cong P\times H$ for a group $H$ if and only if  $C_G(\alpha)$ centralizes $P$. We also prove a generalization of the Frobenius $p$-nilpotency theorem for groups admitting a group of automorphisms of coprime order.
	\end{abstract}

	\section{Introduction}
	All groups considered in this paper are finite. Notation and terminology are standard as in \cite{1}. Let $G$ be a group and $\alpha$ be an automorphism of $G$.   One of natural research topics is to investigate the structure of $G$ under some assumptions on $\alpha$ and the structure of fixed points of $\alpha$ on $G$, denoted by $C_G(\alpha)$. On this regard, Thompson proved the following in his Ph.D thesis:
	\begin{theorem*}\cite[Theorem 1]{12}
	If the order of $\alpha$ is  prime and $C_G(\alpha)=1$, then $G$ is nilpotent.
	\end{theorem*}
	
	If $(|G|,|\alpha|)=1$ then $\alpha$ is called \textbf{a coprime automorphism}.
	Note that in Thompson's theorem, the assumption $C_G(\alpha)=1$ leads that $\alpha$ is necessarily a coprime automorphism since otherwise one can easily observe that $C_G(\alpha)\neq 1$. The above theorem simply says that the existence of a fixed point free automorphism of prime order forces all Sylow subgroups of $G$ to be normal in $G$. Thus, we can intuitively expect that if we have a coprime automorphism $\alpha$ of prime order and $C_G(\alpha)$ is not too big, then some of Sylow subgroups or Hall subgroups of $G$ are normal. This is our motivation to expand Thompson's theorem.
	
	Before stating our main results, we shall mention some preliminary results used repeatedly in this paper and  note some of our conventions: Let $p$ be  a prime number. We say that $G$ is \textbf{$p$-closed} if $G$ has a normal Sylow $p$-subgroup. The group $G$ is said to be \textbf{$p$-nilpotent} if it has a normal Hall $p'$-subgroup. The set of all primes dividing the order of $G$ is denoted by $\pi(G)$. Let $A$ be a group acting on $G$ by automorphisms such that $(|A|,|G|)=1$ and $p\in \pi(G)$. It is well known that there exists an $A$-invariant Sylow $p$-subgroup of $G$ and any two $A$-invariant Sylow $p$-subgroups of $G$ are conjugate by an element of $C_G(A)$  (see \cite[Theorem 3.23 (a) and (b)]{1}). Moreover, each $A$-invariant $p$-subgroup of $G$ is contained in some $A$-invariant Sylow $p$-subgroup of $G$ (see \cite[Corollary 3.25]{1}). For an $A$-invariant normal subgroup $N$ of $G$, write $\overline G=G/N$, the equality $\overline{C_G(A)}=C_{\overline G}(A)$ holds (see \cite[ Corollary 3.28]{1}).

 	\begin{theorem}\label{Second main theorem}
 		Let $\alpha$ be a	coprime automorphism  of $G$ of prime order and let $P$ be an $\alpha$-invariant Sylow $p$-subgroup of $G$. Assume that $p\notin \pi(C_G(\alpha))$. Then $G$ is $p$-nilpotent if and only if $C_{N_G(P)}(\alpha)$ centralizes $P$.
 	\end{theorem}
 	It is easy to observe that $C_{N_G(P)}(\alpha)$ centralizes $P$ if $N_G(P)$ is $p$-nilpotent. The converse of this assertion is also true, which is not so obvious (see Step (1) of the proof of Theorem \ref{Second main theorem}). Thus, we may say that $G$ is $p$-nilpotent if and only if $N_G(P)$ is $p$-nilpotent under the hypothesis of Theorem \ref{Second main theorem}. If we restrict ourself to odd primes, Theorem \ref{Second main theorem} can be obtained by \cite[Theorem A]{7}. Therefore, the main contribution of Theorem \ref{Second main theorem} is for $p=2$. However, we shall not appeal to \cite[Theorem A]{7} in this paper. One of the  main ingredient of the proof of Theorem \ref{Second main theorem} is Theorem \ref{gen. Frob}, which is of independent interest
 	too, as it is a generalization of the Frobenius $p$-nilpotency theorem for groups admitting a group of automorphisms of coprime order. 
 	
 \begin{theorem}\label{gen. Frob}
 	Let $G$ be a group and $A$ be a group acting on $G$ by automorphisms such that $(|A|,|G|)=1$. Assume that for each nontrivial $A$-invariant  $p$-subgroup $U$ of $G$, the group $N_G(U)$ is $p$-nilpotent. Then $G$ is $p$-nilpotent.
 \end{theorem}
	In the case that $A=1$, we have the usual Frobenius $p$-nilpotency theorem (see \cite[Theorem 5.26]{1}).
 	 We also note that Theorem \ref{gen. Frob} shows that Thompson's main result in his Ph.D thesis \cite[Theroem A]{12} can be extended to all primes under suitable changes in the hypothesis. The proof of Theorem \ref {gen. Frob} depends on CFSG, so one might consider whether a  classification free proof is possible. 
 	
 	\begin{theorem}\label{main thm}
 	Let $\alpha$ be a	coprime automorphism  of $G$ of prime order $r$ and let $P$ be an $\alpha$-invariant Sylow $p$-subgroup of $G$ where   $p\in\pi(G)\setminus \pi(C_G(\alpha))$. Assume that $G$ is $PSL(2,2^r)$-free in the case that $p\mid 2^r+1$, and assume that $G$ is $Sz(2^r)$-free in the case that $p\mid 4^r+1$.  Then $G$ is $p$-closed if and only if $C_G(\alpha)$ normalizes $P$.
 	\end{theorem}
	 We would like to note that $C_G(\alpha)$ normalizes $P$ if and only if $G$ has a unique $\alpha$-invariant Sylow $p$-subgroup by \cite[Theorem 3.23 (b)]{1}. So, Theorem \ref{main thm} can be equivalently stated as follows: $G$ has a unique $\alpha$-invariant Sylow $p$-subgroup if and only if $G$ has a unique  Sylow $p$-subgroup under the hypothesis of the theorem.
	
		Lastly, we note that both Theorems A and C are generalizations of Thompson's theorem for different directions, namely one of them gives a necessary and sufficient condition for the existence of a normal Hall $p'$-subgroup of $G$ and  the other one does the same for the existence of a normal Sylow $p$-subgroup of $G$. The following  corollary is a consequence of these results.

	  \begin{corollary}\label{corol.}
	 Let $\alpha$ be a	coprime automorphism  of $G$ of prime order and let $P$ be an $\alpha$-invariant Sylow $p$-subgroup of $G$. Assume that $p\notin \pi(C_G(\alpha))$. Then $G\cong P\times H$ for a group $H$ if and only if  $C_G(\alpha)$ centralizes $P$.
	  \end{corollary}
  
  \section{The proofs of Theorems \ref{Second main theorem} and \ref{gen. Frob}}

   We first prove Theorem \ref{gen. Frob} which we need in  of the proof of Theorem \ref{Second main theorem}.
   
  \begin{lemman}\label{simple Groups}
  	Let $H=PSU(3,2^n)$ and $K=Sz(2^n)$ where $n$ is odd and $n>1$, and let $P$ and $Q$ be Sylow $2$-subgroups of $H$ and $K$, respectively. Then $N_H(P)$ and $N_K(\Omega_1(Q))$ are not $2$-nilpotent.
  \end{lemman}
  \begin{proof}[\textbf{Proof}]
  	Let $G=PGU(3,2^n)$ where $n$ is odd. A characterization of $G$ is given in \cite[Theorem 1]{8} and we see by \cite[Lemma 2]{8} that $G$ has  a T.I. Sylow $2$-subgroup. It follows that $H$ also has a T.I. Sylow subgroup, that is, $P\cap P^h=1$ for $h\in H-N_H(P)$. Then $N_H(P)$ controls $H$-fusion in $P$ by \cite[Thereorem 7.4.6]{3}. Hence, if $N_H(P)$ is $2$-nilpotent, then $H$ is $2$-nilpotent. This is not possible as $H$ is a simple group.
  	
  	Now set $U=\Omega_1(Q)$, which is the group generated by the all involutions of $Q$. Note that $Q$ is called a Suzuki-$2$-group and studied in \cite{9}. In this paper, it is showed that $U=Z(Q)$, and hence $U$ is weakly closed  in $Q$. If $N_K(U)$ is $2$-nilpotent, then $K$ is not simple by Gr\"{u}n's Thereom (see \cite[Theorem 7.4.5]{3}). Thus, $N_K(U)$ is not $2$-nilpotent.
  \end{proof}
  \begin{proof}[\textbf{Proof of Theorem \ref{gen. Frob}}]
  	Let $G$ be a minimal counterexample to the theorem and let $P$ be an $A$-invariant Sylow $p$-subgroup of $G$. We first show that $p=2$ by the Thompson $p$-nilpotency theorem. First assume that $p$ is odd. 
  	Let $J(P)$ denote the Thompson subgroup of $P$. Note that $J(P)$ is characteristic in $P$ by \cite[Lemma 7.2]{1}. Then we see that  $J(P)$ and $Z(P)$ are $A$-invariant, and so both $N_G(J(P))$ and $N_G(Z(P))$ are $p$-nilpotent by the hypothesis. It follows that $G$ is $p$-nilpotent by the Thompson $p$-nilpotency theorem \cite[Theorem 7.1]{1}. Thus, we see that $p=2$. Note also that we may assume $O_2(G)=1$, since otherwise there is nothing to prove. 
  	
  	Assume that $O_{2'}(G)\neq 1$ and set $\overline G=G/O_{2'}(G)$. Let $X$ be a nontrivial $A$-invariant $2$-subgroup of $\overline G$ and let $V$ be the full inverse image of $X$ in $G$. Then $V$ is $A$-invariant and $V/O_{2'}(G)=X$, which is a $2$-group. Let $U$ be an $A$-invariant Sylow $2$-subgroup of $V$. We have $V=UO_{2'}(G)$, and so $\overline U=X$.  Note that $N_G(U)$ is $2$-nilpotent by  the hypothesis. We see that $\overline{N_G(U)}=N_{\overline G}(\overline U)=N_{\overline G}(X)$
  	by \cite[Lemma 2.17]{1}. It follows that $N_{\overline G}(X)$ is $2$-nilpotent as a homomorphic image of a $2$-nilpotent group. Thus, $\overline G$ satisfies the hypothesis, so $\overline G=G/O_{2'}(G)$ is $2$-nilpotent by the inductive argument. This yields that $G$ is $2$-nilpotent, which is a contradiction. Thus, $O_{2'}(G)=1$ as desired.
  	
  	Now let $N$ be a proper $A$-invariant normal subgroup of $G$. Clearly, $N$ satisfies the hypothesis, and hence $N$ is $2$-nilpotent by the minimality of $G$. Let $H$ be a normal Hall $2'$-subgroup of $N$. Then $H\lhd G$ as $H$ is characteristic in $N$, and so $H\leq O_{2'}(G)=1$, which forces that $N$ is a $2$-group. Next we see that $N\leq O_2(G)=1$ and $G$ is characteristically simple.
  	
  	Consequently, $G = G_1 \times G_2 \times \dots \times G_n$ where $G_i$ are isomorphic simple groups and
  	$A $ acts transitively on $\{ G_i
  	\mid i = 1, . . . , n\}.$ Assume $n > 1$ and set $B=Stab_A(G_1)$. Let $U$ be a nontrivial $B$-invariant $2$-subgroup of $G_1$ and let $T=\{t_1,t_2,...,t_n\}$ be a right transversal set for $B$ in $A$. Without loss of generality, assume $t_1=1$ and $G_1^{t_i}=G_i$ for $i=1,2,...,n$. It follows that $U^{t_i}\leq G_i$, and so $A$ acts transitively on $\{U^{t_i} \mid i=1,2,...,n \}$. Thus, we obtain that $V=U^{t_1}\times U^{t_2}\times \dots \times U^{t_n}$ is $A$-invariant, and so $N_G(V)=\prod_{i=1}^{n}N_{G_i}(U^{t_i})$ is $2$-nilpotent by the hypothesis. In particular, we get that $N_{G_1}(U)$ is $p$-nilpotent, that is, the pair $(B,G_1)$ satisfies the hypothesis. By the minimality of $G$, we obtain that $G_1$ is $2$-nilpotent. This contradiction shows that $n=1$, that is, $G$ is simple.
  	
  	We see from \cite[Table 5]{2} that $G$ is a simple group of Lie type and  $G\cong \! ^d\!\sum(q^r)$ and $C_G(A)\cong \!^d\!\sum(q)$ for a root system $\sum$ and a prime power $q$ where $r=|A|$  by (\cite[Theorem 4.9.1 (a) and (c)]{3}). We also see that $2\in \pi(C_G(A))$ by their known orders. Since each $2$-subgroup of $C_G(A)$ is $A$-invariant, we obtain that $C_G(A)$ is $2$-nilpotent by the Frobenius theorem. In particular, we obtain that $C_G(A)$ is solvable. By the list of simple groups of Lie types, we see that only possibility for solvable $C_G(A)$ are the following groups; $PSU(3,2)$, $PSL(2,2)$, $PSL(2,3)$ and $Sz(2)$.
  	Note that $PSL(2,3)\cong A_4$, the alternating group on $4$-letters, and it is not $2$-nilpotent. The rest of the groups are $2$-nilpotent, which can be checked easily. Thus, $G$ is isomorphic to the one of the groups; $PSU(3,2^r)$, $PSL(2,2^r)$, and $Sz(2^r)$ where $r$ is odd. Clearly, $\Omega_1(P)$ is $A$-invariant since it is characteristic in $P$. It follows that both $N_G(P)$ and $N_G(\Omega_1(P))$ are $2$-nilpotent by the hypothesis of our theorem.
  	Hence, we see that $G\neq PSU(3,2^r)$ and $G\neq Sz(2^r)$ by Lemma \ref{simple Groups}, and so $G=PSL(2,2^r)$. In this case, we see that $P$ is abelian by \cite[Theorem 8.6.3 (b)]{11}. Since $N_G(P)$ is $2$-nilpotent, we get that $G$ is $2$-nilpotent by the Burnside $p$-nilpotency theorem (see \cite[Theorem 7.4.3]{3}). 
  	This final contradiction completes the proof.  
  \end{proof}
  \raggedbottom
  
   The following representation theoretic fact is used in the proof of Theorem \ref{Second main theorem}.
  
  \begin{lemman}\cite[Theorem 3.4.4]{6}\label{repthm}
  	Let $A=PQ$ be a group such that $Q$ is an elementary abelian $q$-group and $P$ is  a cyclic group of prime order $p$. Assume that $Q$ is a minimal normal subgroup of $A$ and $C_A(Q)=Q$. Let $V$ be a faithful $FA$-module such that the characteristic of the filed $F$ is coprime to $|A|$. Then $C_V(P)\neq 0$.
  \end{lemman}

  \begin{proof}[\textbf{Proof of Theorem \ref{Second main theorem}}]
  	Let $P$ be an $\alpha$-invariant Sylow $p$-subgroup of $G$ where $p\notin \pi(C_G(\alpha))$. It is obvious that if $G$ is $p$-nilpotent, then the $p'$-subgroup $C_{N_G(P)}(\alpha)$ centralizes $P$. Thus, we only show that if $C_{N_G(P)}(\alpha)$ centralizes $P$, then $G$ is $p$-nilpotent. Let $G$ be a minimal counterexample to the theorem. We shall derive a contradiction over a series of steps.\vspace{0.13 cm}
  	
  	$\textbf{1)}$  $N_G(P)$ is $p$-nilpotent.
  	\vspace{0.13 cm}
  	
  	We have $C_{N_G(P)}(\alpha)\leq C_G(P)$ by the hypothesis. Note also that both $N_G(P)$ and $C_G(P)$ are $\alpha$-invariant. Then we see that $C_{N_G(P)/C_G(P)}(\alpha)=1$, and so $N_G(P)/C_G(P)$ is nilpotent by the Thompson theorem. Appealing to the Schur-Zassenhaus theorem \cite[Theorem 3.8]{1}, we see that there exists a Hall $p'$-subgroup $H$ of $N_G(P)$. Thus, we obtain that $[H,P]\leq C_G(P)$ as $N_G(P)/C_G(P)$ is nilpotent. On the other hand, $H$ normalizes $P$, which yields that $[H,P]\leq P\cap C_G(P)=Z(P)$. It follows that $[H,P,P]=1$, and so $[H,P]=1$ due to the coprimeness (see \cite[Lemma 4.29]{1}). Then we have $H\lhd N_G(P)$.\vspace{0.13 cm}
  	
  	$\textbf{2)}$ $O_{p'}(G)$ is trivial.\vspace{0.13 cm}

  	Assume that $O_{p'}(G)\neq 1$ and set $\overline G=G/O_{p'}(G)$. Note that $\overline {N_G(P)}=N_{\overline G}(\overline P)$ by \cite[Lemma 2.17]{1}, and so we obtain that $N_{\overline G}(\overline P)$ is $p$-nilpotent by Step (1).  We have $\overline{C_G(\alpha)}=C_{\overline G}(\alpha)$ due to the coprimeness, and hence $p\notin \pi(C_{\overline G}(\alpha)).$ It follows that $C_{N_{\overline G}(\overline P)} (\alpha)$ is a $p'$-group, which yields $C_{N_{\overline G}(\overline P)} (\alpha)$ centralizes $\overline P$ as  $N_{\overline G}(\overline P)$ is both $p$-nilpotent and $p$-closed. Consequently, $\overline G$ satisfies the hypothesis, and hence we obtain that $\overline G=G/O_{p'}(G)$ is $p$-nilpotent by the minimality of $G$. It follows that $G$ is $p$-nilpotent, which is a contradiction. Hence, we get $O_{p'}(G)=1$ as desired.\vspace{0.13 cm}
  	
  	$\textbf{3)}$ $O_p(G)\neq 1$ and $G/O_p(G)$ is $p$-nilpotent.\vspace{0.13 cm}
  	
  	Since $G$ is not $p$-nilpotent, there exists a nontrivial $\alpha$-invariant $p$-subgroup $U$ of $G$ such that $N_G(U)$ is not $p$-nilpotent by Theorem \ref{gen. Frob}. Among such $\alpha$-invariant $p$-subgroups, choose $U$ of maximal possible order. Clearly, $N_G(U)$ is also $\alpha$-invariant and we may choose an $\alpha$-invariant Sylow $p$-subgroup $V$ of $N_G(U)$. Without loss of generality, we may assume that $V$ is contained in $P$, that is, $V=N_P(U)$. Note that $U<P$ as $N_G(P)$ is $p$-nilpotent by Step (1), and so $U<V$. The maximality of $U$ forces that $N_G(V)$ is $p$-nilpotent. Write $H=N_G(U)$ and assume that $H<G$. Clearly, $p\notin \pi(C_H(\alpha))$. Moreover, the $p'$-group $C_{N_H(V)}(\alpha)$ centralizes $V$, as $N_H(V)\leq N_G(V)$ is $p$-nilpotent. Then we see that $H$ satisfies the hypothesis, and so $H=N_G(U)$ is $p$-nilpotent by the minimality of $G$. This contradiction shows that $H=G$, that is, $U\lhd G$. It follows that $1<U\leq O_p(G)$, and in particular $ O_p(G)\neq 1$. It is routine to see that $G/O_p(G)$ satisfies the hypothesis, and hence $G/O_p(G)$ is $p$-nilpotent by the minimality of $G$.\vspace{0.13 cm}
  	
  	$\textbf{4)}$ $P$ is a maximal $\alpha$-invariant subgroup of $G$.\vspace{0.13 cm}
  	
  	Let $M$ be  a proper $\alpha$-invariant subgroup of $G$ such that $P\leq M$. Clearly, $M$ satisfies the hypothesis, and so $M$ is $p$-nilpotent. Let $K$ be a normal Hall $p'$-subgroup of $M$. It follows that $[K,O_{p}(G)]\leq K\cap O_{p}(G) =1$, that is, $K\leq C_G( O_{p}(G))$. Note that $G$ is a $p$-separable group by Step (3). Since $O_{p'}(G)=1$ by Step (2),  we obtain that that $C_G( O_{p}(G))\leq O_p(G)$ by the Hall-Higman theorem (see \cite[Theorem 3.21]{1}), and so $K=1$. Then $M$ is a $p$-group, which forces that $M=P$.\vspace{0.13 cm}
  	
  	$\textbf{5)}$ Final contradiction.\vspace{0.13 cm}
  	
  	Write $\overline G=G/O_p(G)$.
  	It follows by Step (4) that $\overline P$ is a maximal $\alpha$-invariant subgroup of $\overline G$. Note that $G$ has a Hall $p'$-subgroup as $G$ is $p$-separable by Step (3). Let $X$ be an $\alpha$-invariant Hall $p'$-subgroup of $ G$. We have that $\overline X\lhd \overline G$ by Step (3). Now consider the coprime action of $\langle \alpha \rangle \overline P$ on $\overline X$. Note that we may choose an $\alpha$-invariant Sylow $q$-subgroup $Q$ of $X$ for some $q\in \pi( X)$ such that $\overline Q$ is  $\langle \alpha \rangle \overline P$-invariant. Then the group $\overline {PQ}$ is an $\alpha$-invariant subgroup of $\overline G$. Thus, we get that $\overline G= \overline {PQ}$ by the maximality of $\overline{P}$, which yields that $X=Q$. Since $\overline P$ $\Phi(\overline Q)$ is also $\langle \alpha \rangle \overline P$-invariant, we may similarly conclude that $\Phi(\overline Q)=1$ by using the maximality of $\overline P$. As a consequence, we obtain that $Q\cong \overline Q$ is an elementary abelian $q$-group. Since $p\notin \pi(C_G(\alpha))$, we get that $C_G(\alpha)$ is a $q$-group. Then we see that $C_G(\alpha)\leq Q$. If $Q=C_G(\alpha)$, then $P=[G,\alpha]\lhd G$ , which is a contradiction by Step (1). Thus, $[Q,\alpha]\neq 1$. As $Q$ is abelian, we have that $Q=[Q,\alpha]\times C_Q(\alpha)$ by \cite[Theorem 4.34]{1}.   In particular, we have that $C_{[Q,\alpha]}(\alpha)=1$. Let $U$ be a minimal $\alpha$-invariant subgroup of $[Q,\alpha]$. Now consider the coprime action of $A:=\langle\alpha \rangle U$ on $V:=O_p(G)/\Phi(O_p(G))$. Then we may regard $V$ as an $FA$-module where $F=\mathbb Z_p$. Note that $C_V(\alpha)=0$ since $C_P(\alpha)=1$. Moreover, if $u$ acts trivially on $V$ for some $u\in U$, then $u\in C_G(O_p(G))$ by \cite[Corollary 3.29]{1}.  We obtain that $u=1$ as $ C_G(O_p(G))\leq O_p(G)$.  It follows that $V$ is a faithful $FA$-module. However, we get that $C_V(\alpha)\neq 0$ by Lemma \ref{repthm}. This final contradiction completes the proof.
  \end{proof}

\section{The proofs of Theorem \ref{main thm} and Corollary \ref{corol.}}

\begin{propositionn}\label{main lemma}
		Let $\alpha$ be a	coprime automorphism  of $G$ of prime order and let $P$ be an $\alpha$-invariant Sylow $p$-subgroup of $G$. Assume that $p\notin \pi(C_G(\alpha))$ and $G$ is $p$-separable. If $C_G(\alpha)$ normalizes $P$, then $G$ is $p$-closed.
\end{propositionn}
	
	\begin{proof}[\textbf{Proof}]
		Let $G$ be a minimal counterexample to the proposition. Let $P$ be an $\alpha$-invariant Sylow $p$-subgroup of $G$. Suppose that we have a nontrivial proper $\alpha$-invariant subgroup $N$ of $G$. An $\alpha$-invariant Sylow $p$-subgroup of $N$ must be contained in $P$ as it is the unique $\alpha$-invariant Sylow $p$-subgroup of $G$. This forces that $P\cap N$ is an $\alpha$-invariant Sylow $p$-subgroup of $N$.  Clearly, $C_N(\alpha)$ normalizes $P\cap N$. Thus, we see that $N$ satisfies the the hypothesis of the proposition as $p\notin \pi(C_N(\alpha))\subseteq \pi(C_G(\alpha))$. Now assume that $N$ is also normal in $G$ and set $\overline G=G/N$. Due to the coprimeness, we have $C_{\overline G}(\alpha)=\overline{C_G(\alpha)}$, and hence we see that $C_{\overline G}(\alpha)$ normalizes $\overline P$ and  $p\notin \pi(C_{\overline G}(\alpha))$. Thus, both $N$ and $G/N$ are $p$-closed by the minimality of $G$. We shall use these observations throughout the proof.
		
		Assume first that $O_p(G)\neq 1$. It follows  by the first paragraph that $P/O_p(G)\lhd G/O_p(G)$. Then we get $P\lhd G$, which is not the case. We see that $O_p(G)=1$, and so $O_{p'}(G)\neq 1$ since $G$ is $p$-separable. 
		
		Consider the group $T=PO_{p'}(G)$, which is clearly $A$-invariant. Suppose $T<G$. Since $T$ satisfies the hypothesis, we get $P\lhd T$ by the minimality of $G$ and so $[P,O_{p'}(G)]=1$. Then we get $P\leq C_G(O_{p'}(G))\leq O_{p'}(G)$. The last inequality follows by the Hall-Higman theorem (see \cite[Theorem 3.21]{1}). As result we get $P=1$, which is a contradiction. Thus, we have $G=T$.
		 
		 Consider the coprime action of $\langle \alpha \rangle P$ on $O_{p'}(G)$. Let $Q$ be  an $\langle \alpha \rangle P$ -invariant Sylow $q$-subgroup of $O_{p'}(G)$ for a prime $q\in \pi(O_{p'}(G))$.  Assume that $Q\neq O_{p'}(G)$. Then $P$ is normal in $PQ$ by the inductive hypothesis, and hence $P$ acts trivially on $Q$.  We get that $[P,O_{p'}(G)]=1$ since $q$ is arbitrary, which yields that $P\lhd PO_{p'}(G)=G$. This contradiction shows that $Q=O_{p'}(G)$.
		 
		  Assume that $\Phi(Q)\neq 1$. Then $P\Phi(Q)/\Phi(Q)$ is normal in $G/\Phi(Q)$ by induction applied to $G/\Phi(Q)$. As a result, $P\Phi(Q)$ is normal in $G$. On the other hand,  $P$ is normal in $P\Phi(Q)$ by induction applied to $P\Phi(Q)$. It follows that $P$ is normal in $G$ as $P$ is characteristic in $P\Phi(Q)$, which is a contradiction. So, we get that $\Phi(Q)=1$, that is, $Q$ is elementary  abelian. Clearly, $\Phi(P)Q$ is an $\alpha$-invariant proper subgroup of $G$. It follows that $\Phi(P)\lhd \Phi(P)Q$ by induction, and so $[\Phi(P),Q]\leq \Phi(P)\cap Q=1$. Thus, we obtain that $\Phi(P)\leq C_G(Q)\leq Q$, which yields that $\Phi(P)=1$. Similarly, we observe that each proper $\alpha$-invariant subgroup of $P$ is trivial, that is, $P$ is a minimal $\alpha$-invariant subgroup of $G$. 
		  
		   Note that $C_G(\alpha)$ is a $q$-group since $p\notin \pi(C_G(\alpha))$, and hence $C_G(\alpha)\leq Q$. By the fact that $C_G(\alpha)$ normalizes $P$ and $Q\lhd G$, we obtain  $$[C_G(\alpha),P]\leq P\cap Q=1.$$ It follows that $C_G(\alpha)\lhd G$ since $Q$ is abelian. Now we claim that $C_G(\alpha)\neq 1$. Consider the coprime action of $A:=\langle \alpha \rangle P $ on $Q$. Note that $C_A(P)=P$ and $A$ acts faithfully on $Q$. We get that $C_Q(\alpha)\neq 1$ by Lemma \ref{repthm}, that is, $C_G(\alpha)\neq 1$. Note also that $Q\neq C_G(\alpha)$ as $[P,Q]\neq 1$, and so $PC_G(\alpha)$ is proper in $G$. We obtain that $PC_G(\alpha)/C_G(\alpha)$ is normal in $G/C_G(\alpha)$ by induction applied to $G/C_G(\alpha)$. Next we see that $PC_G(\alpha)$ is normal in $G$, and so we get $P\lhd G$, which is the final contradiction.  
	\end{proof}
	
\begin{lemman}\label{clasification}
	Let $G$ be a nonabelian simple group and let $\alpha$ be a coprime automorphism of $G$ of prime order $r$. Pick an $\alpha$-invariant $P\in Syl_p(G)$ for $p\in \pi(G)\setminus(\pi(C_G(\alpha)))$. If $C_G(\alpha)$ normalizes $P$, then one of the following holds;
	
	\begin{enumerate}
		\item[a)] $G=PSL(2,2^r)$ and  $r\geq 5$. Moreover, $p\geq 5$ and  $p$ is a divisor of $2^r+1$.
		
		\item[b)] $G=Sz(2^r)$ and $r\geq 7$. Moreover, $p\geq 7$ and  $p$ is a divisor of $2^r\pm \sqrt{2^{r+1}}+1$ where the sign $\pm$ is chosen such that $5$ divides $2^r\pm \sqrt{2^{r+1}}+1$.
	\end{enumerate} 
\end{lemman}
\raggedbottom
\begin{proof}[\textbf{Proof}]
	We observe from \cite[Table 5]{2} that only nonabelian simple groups which admit a
	nontrivial coprime action are simple groups of Lie type defined over some finite field $F$ and $\large \langle \alpha \rangle$ is a group of automorphisms induced by automorphisms of the underlying field. Note that $G\cong \! ^d\!\sum(q^r)$ and $C_G(\alpha)\cong \!^d\!\sum(q)$ for a root system $\sum$ and a prime power $q$ where $r=|\alpha|$  by \cite[Theorem 4.9.1 (a) and (c)]{3}. 
	
	Now set $C=C_G(\alpha)$. We see that $C$ is guaranteed to be a maximal subgroup of $G$ by \cite[Theorem 1]{4} except for some minimal nonsolvable groups $PSL(2,2^r)$, $PSL(2,3^r)$ and $Sz(2^r)$. We  also observe from  \cite[Table 8.1]{5} that $C$ is also a maximal subgroup of $G$ when $G=PSL(2,3^r)$. In the case that $C$ is a maximal subgroup of $G$, the equality  $PC=G$ holds since $p\notin \pi(C)$ and $C$ normalizes $P$ by the hypothesis. It follows that $P\lhd G$, which is not possible as $G$ is simple. This argument shows that $G=PSL(2,2^r) \ or \ Sz(2^r)$.
	
	First suppose that $G=PSL(2,2^r)$ and set $q=2^r$. Then we see that $C=PSL(2,2)\cong S_3$. Note that $p\geq 5$ and $|\alpha|=r\geq 5$ as $p,r\notin \pi(C)$. We see that $C$ is contained in a maximal subgroup $D$, which is a dihedral group of order $2(q+1)$ (see \cite[Table 8.1]{5}). Let $A$ be the subgroup of $C$ of order $3$. Clearly, $A$ is normalized by $D$, and so $D=N_G(A) $ as $G$ is simple and $D$ is a maximal subgroup of $G$. Now we claim that $p \mid q+1$. Since $p\neq 2$, we have that $P$ is cyclic by \cite[Theorem 8.6.9]{11}, and so $Aut(P)$ is abelian. It follows that $C/C_{C}(P)$ is abelian. Since $A=C'$, we get that $A\leq C_{C}(P) $, and so $P\leq C_G(A)\leq N_G(A)=D$. Since $p$ is odd and $|D|=2(q+1)$, we have that $p$ divides $q+1$. Then the claim follows.  Consequently, we observe that if such a Sylow $p$-subgroup of $G$ exist, it must be included in $D=N_G(A)$. On the other hand, $D=N_G(A)$ is $\alpha$-invariant  and $\pi(D)\neq \{2,3\}$ as $r\geq 5$. Pick an $\alpha$-invariant Sylow $p$-subgroup $P$ of $D$ for $p\geq 5$. We see that $P$ is also a Sylow subgroup of $G$ as $|G|=(q-1)q(q+1)$. Clearly, $P$ is normalized by $C$ and $p\notin \pi(C)$, which completes the proof for this case.	 
	
	Now we shall investigate the case where $G=Sz(2^r)$. Note that $C\cong Sz(2)$, which is a Frobenius group of order $20$. Note also that $G$ is of order $q^2(q^2+1)(q-1)$ and $(q-1,q^2+1)=1$.  We see  from \cite[Table 8.16]{5} that $G$ has  four maximal subgroups up to conjugacy of orders $q^2(q-1),2(q-1),(q-\sqrt{2q}+1)4$ and $(q+\sqrt{2q}+1)4$. Since $5$ is a divisor of $q^2+1$, we see that $5$ divides either $q-\sqrt{2q}+1$ or $q+\sqrt{2q}+1$.  Notice that there is no proper subgroup whose order is divisible by both $5$ and a prime dividing $q-1$ due to  the orders of the maximal subgroups. Thus, the proper subgroup $PC$ must be contained in a maximal subgroup $M$ of order $(q\pm\sqrt{2q}+1)4$ where the sign $\pm$ is chosen such that $5$ divides $q\pm \sqrt{2q}+1$. Clearly, $p,r\notin \{2,5\}$ by the hypothesis. We see that $p\neq 3$ as $|G|$ is not divisible by $3$, and so  $p\geq 7$. If $r=3$ then  $|M|=20$, and so $P=1$. Thus, we have $r\geq7$. Let $B$ be a normal subgroup of $C$ of order $5$. Since $M$ is an extension of a cyclic group by $\mathbb Z_4$, we obtain that $M=N_G(B)$, and so $M$ is $\alpha$-invariant. We claim that $\pi(M)\neq \{2,5\}$. If $|M|\equiv 0 \ mod \ 5^2$, then we have $q^2=4^r\equiv -1 \ mod \ 5^2$. It follows that $r\equiv 5 \ mod \ 10$, which forces that $r=5$ as $r$ is a prime. This is not possible as $r\geq 7$, and so $\pi(M)\neq \{2,5\}$ as desired.  Thus, $M$ contains an  $\alpha$-invariant Sylow $p$-subgroup of $G$ for $p \geq 7$, which is normalized by $C$. 
\end{proof}

\begin{proof}[\textbf{Proof of Theorem \ref{main thm}}]
Let $P$ be an $\alpha$-invariant Sylow $p$-subgroup of $G$ where $p\notin \pi(C_G(\alpha))$.
If $P\lhd G$, then clearly $P$ is normalized by $C_G(\alpha)$. Thus, we need to prove the reverse direction, that is, if $C_G(\alpha)$ normalizes $P$, then $P\lhd G$.	Let $G$ be a minimal counterexample to the theorem and 
let $N$ be a proper $\alpha$-invariant normal subgroup of $G$. Suppose that $N\neq 1$. It is routine to see that both $N$ and $G/N$ satisfy the hypothesis (see the first paragraph of the proof of Proposition \ref{main lemma}). Note also that if $G$ is $X$-free for a group $X$, then both $N$ and $G/N$ are clearly $X$-free. Thus, we see that both $N$ and $G/N$ are $p$-closed. It follows that $G$ is $p$-separable, and so $G$ is $p$-closed by Proposition \ref{main lemma}. This contradiction shows that $N=1$ and $G$ is characteristically simple.

Then we see that $G = G_1 \times G_2 \times \dots \times G_n$ where $G_i$ are isomorphic simple groups and
$\langle \alpha \rangle $ acts transitively on $\{ G_i
	\mid i = 1, . . . , n\}.$ Assume $n > 1$. It follows that $n=r$ and $$C_G(\alpha)=\{(g_0,g_1,...,g_{r-1}) \mid   g_j\alpha=g_{j+1} \ \emph{for all} \ j\in \mathbb Z_r  \}.$$
	
	Hence, we get that $C_G(\alpha)\cong G_i$ for $i=1, ...,n$, which yields that $\pi(C_G(\alpha))=\pi(G)$. Then a Sylow $p$-subgroup of $G$ is trivial, which is a contradiction. Consequently, we see that $n=1$, that is, $G$ is a simple group. It follows that either $G=PSL(2,r)$ and $p\mid 2^r+1$, or $G=Sz(2^r)$ and $p\mid 4^r+1$ by Lemma \ref{clasification}, which are both impossible by our hypothesis.
\end{proof}

\begin{remark}
	Theorem B can be stated in a more general way by using the constraints in Lemma \ref{clasification} about the primes $p$ and $r$. For example, if $p=3$ or $r=3$ there is no need to assume that $G$ is $PSL(2,2^r)$ or $Sz(2^r)$ free.
\end{remark}

\begin{proof}[\textbf{Proof of Corollary \ref{corol.}}]
	First assume that $C_G(\alpha)$ centralizes $P$ and set $N=N_G(P)$.
	Since $C_N(\alpha)\subseteq C_G(\alpha)$ centralizes $P$, we obtain that $G$ is $p$-nilpotent by Theorem A, and in particular, $G$ is $p$-separable. It follows that $G$ is $p$-closed by Proposition \ref{main lemma} as $C_G(\alpha)$ also normalizes $P$. Thus, $G\cong P\times H$ where $H$ is a Hall $p'$-subgroup of $G$.
	
	Now suppose that $G\cong P\times H$ for a group $H$. Then $G$ has a normal Hall $p'$-subgroup $N$. We see that $C_G(\alpha)\subseteq N$ as $p\notin \pi(C_G(\alpha))$. Thus, we get that $C_G(\alpha)$ centralizes $P$ as $P \unlhd G$.
\end{proof}

\section{A Question about the Fitting height of a group}
Let $G$ be a solvable group and $A$ be a group acting on $G$ by automorphisms such that $(|A|,|G|)=1$. Let $F(G)$ denote the Fitting subgroup of $G$ and set $F_0(G)=1$, $F_1(G)=F(G)$ and define inductively $F_{i+1}(G)$ as the full inverse image of $F(G/F_i(G))$ in $G$. The smallest natural number $n$ such that $F_n(G)=G$ is called \textbf{the Fitting height} of $G$. 

Let $l(A)$ denote the number of not necessarily distinct primes
whose product is $|A|$. If  $l(A)=1$ and $P$ is an $A$-invariant Sylow $p$-subgroup of $G$ such that $p\notin C_G(A)$, then $P\leq F(G)$ by Theorem \ref{main thm}. Thus, it is natural to ask the following question:

\begin{questionn}\label{question}
Let $G$ be a solvable group and $A$ be a group acting on $G$ by automorphisms such that $(|A|,|G|)=1$. Let $P$ an $A$-invariant Sylow $p$-subgroup of $G$ such that $p\notin \pi (C_G(A))$. Assume that  $C_G(A)$ normalizes $P$. Is there a natural number $n$ which only depends on $l(A)$ such that $P\leq F_{n}(G)$?
\end{questionn}
Assume the hypothesis and the notations of Question \ref{question}:
\begin{conjecturen}\label{con1}
   $P$ is contained in $ F_n(G)$ where $n=l(A)$.
\end{conjecturen}
\begin{conjecturen}\label{con2}
 Assume further that $C_G(A)$ centralizes $P$.  Then  $P$ is contained in $ F_n(G)$ where $n=l(A)$.
\end{conjecturen}

Clearly, Conjecture $\ref{con1}$ is stronger than Conjecture $\ref{con2}$. On the other hand, if Conjecture $\ref{con2}$ is true than a well known conjecture about the fitting height can be verified as a corollary, which states that  the Fitting height of $G$ is bounded by $l(A)$ if $C_G(A)=1$. For results related to this conjecture, see \cite{10}. Some similar conjectures could be made about the $p$-length of $p$-separable groups.

 \section*{Acknowledgements}
I would like to thank Prof. Danila Revin for his help in the proof of Lemma \ref{clasification}.

	\end{document}